\documentclass[11pt]{amsart}
\usepackage{amsmath}
\usepackage[active]{srcltx}
\usepackage{t1enc}
\usepackage[latin2]{inputenc}
\usepackage{verbatim}
\usepackage{amsmath,amsfonts,amssymb,amsthm}
\usepackage[mathcal]{eucal}
\usepackage{enumerate}
\usepackage[centertags]{amsmath}
\usepackage{graphics}

\newtheorem{theorem}{Theorem}

\newtheorem{lemma}{Lemma}

\newtheorem{corollary}{Corollary}

\newtheorem*{FS}{Theorem F}

\begin{document}
\author{Ushangi Goginava and K\'aroly Nagy}
\address{U. Goginava, Department of Mathematics, Faculty of Exact and Natural
Sciences, Tbilisi State University, Chavchavadze str. 1, Tbilisi 0128,
Georgia}
\email{zazagoginava@gmail.comm}
\title[Strong Approximation]{Strong Approximation by Marcinkiewicz Means of
two-dimensional Walsh-Kaczmarz-Fourier Series}
\address{K. Nagy, Institute of Mathematics and Computer Sciences,
University of Ny\'\i regyh\'aza,
S\'ost\'oi str. 31/B,
4400 Ny\'\i regyh\'aza,
Hungary}
\email{nagy.karoly@nye.hu}
\date{}
\maketitle

\begin{abstract}
In this paper we study the exponential uniform strong approximation of
Marcinkiewicz type of two-dimensional Walsh-Kaczmarz-Fourier series. In
particular, it is proved that  the Marcinkiewicz type of two-dimensional
Walsh-Kaczmarz-Fourier series of the continuous function $f$ is uniformly
strong summable to the function $f$ exponentially in the power $1/2$.
Moreover, it is proved that this result is best possible.
\end{abstract}

\footnotetext{%
2000 Mathematics Subject Classification 42C10 .
\par
The research of first author was supported by Shota Rustaveli National Science Foundation
grant no.DI/9/5-100/13 (Function spaces, weighted inequalities for integral
operators and problems of summability of Fourier series)}

It is known that there exist continuous functions the trigonometric (Walsh)
Fourier series of which do not converge uniformly. However, as it was proved by Fej\'er 
\cite{Fe} in 1904, the arithmetic means of the differences between the
function and its Fourier partial sums converge uniformly to zero. The
problem of strong summation was initiated by Hardy and Littlewood \cite{HL}.
They generalized Fej\'er's result by showing that the strong means also
converge uniformly to zero for any continuous function. The investigation of
the rate of convergence of the strong means was started by Alexits \cite{AK}. Many papers have been published which are closely related with strong
approximation and summability. We note that a number of significant results
are due to Leindler \cite{Le1,Le2,Le3}, Totik \cite{To1,To2,To3}, Fridli and
Schipp \cite{FS2}, Gogoladze \cite{Go}, Goginava, Gogoladze, Karagulyan \cite{GGKCA}. Leindler has also published a monograph \cite{Le4}.

The results on strong summation and approximation of trigonometric Fourier
series have been extended for several other orthogonal systems. For
instance, concerning the Walsh system see Schipp \cite{Sch1,Sch2,Sch3},
Fridli, Schipp \cite{FS,FS2}, Fridli \cite{kacz}, Rodin \cite{Ro1},
Goginava, Gogoladze \cite{GGSMH, GGCA}, G\'at, Goginava, Karagulyan \cite{GGKAM,GGKJMAA}, Goginava, Gogoladze, Karagulyan \cite{GGKCA} and concerning
the Ciesielski system see Weisz \cite{We1,We2}. The summability of multiple
Walsh-Fourier series have been investigated in \cite{GN,GoAM, GoSM, N,
Webook2}.

Fridli \cite{kacz} proved that the following theorem is true.

\begin{FS}
Let $\psi $ be monotonically increasing function defined on $[0,\infty )$
for which $\lim_{u\rightarrow 0+}\psi (u)=0$. Then 
\begin{equation*}
\lim\limits_{n\rightarrow \infty }\frac{1}{n}\sum\limits_{k=1}^{n}\psi
\left( \left\vert S_{k}^{\kappa}(f;x)-f(x)\right\vert \right)
=0\quad (f\in C(G))
\end{equation*}%
if and only if there exists $A>0$ such that $\psi (t)\leq \exp (At)$ $(0\leq
t<\infty )$. Moreover, the convergence is uniform in $x$.
\end{FS}

In this paper we study the exponential uniform strong approximation of the
Marcinkiewicz means of the two-dimensional Walsh-Kaczmarz-Fourier series.
In particular, it is proved that the Marcinkiewicz type of the
two-dimensional Walsh-Kaczmarz-Fourier series of the continuous function $f$
is uniformly strong summable to the function $f$ exponentially in the power $%
1/2$. Moreover, it is proved that this result is best possible.

\section{Walsh functions}

Let $\mathbb{P}$ denote the set of positive integers, $\mathbb{N}:=\mathbb{P}%
\cup \{0\}.$ Denote $\mathbb{Z}_{2}$ the discrete cyclic group of order 2,
that is $\mathbb{Z}_{2}=\{0,1\},$ where the group operation is the modulo 2
addition and every subset is open. The Haar measure on $\mathbb{Z}_{2}$ is
given such that the measure of a singleton is 1/2. Let $G$ be the complete
direct product of the countable infinite copies of the compact groups $%
Z_{2}. $ The elements of $G$ are of the form $x=\left(
x_{0},x_{1},...,x_{k},...\right) $ with coordinates $x_{k}\in \{0,1\}\left(
k\in \mathbb{N}\right) .$ The group operation on $G$ is the coordinate-wise
addition, the measure (denoted by $\mu $) and the topology are the product
measure and topology. The compact Abelian group $G$ is called the Walsh
group. A base for the neighbourhoods of $G$ can be given in the following
way \cite{SWS}: 
\begin{equation*}
I_{0}\left( x\right) :=G,
\end{equation*}%
\begin{equation*}
I_{n}\left( x\right) :=I_{n}\left( x_{0},...,x_{n-1}\right) :=\left\{ y\in
G:\ y=\left( x_{0},...,x_{n-1},y_{n},y_{n+1},...\right) \right\} ,
\end{equation*}%
$\left( x\in G,n\in \mathbb{N}\right) $.
These sets are called dyadic intervals. Let $0=\left( 0:i\in \mathbb{N}%
\right) \in G$ denote the null element of $G,$ $I_{n}:=I_{n}\left( 0\right) $
$\left( n\in \mathbb{N}\right) $. Set $e_{n}:=\left( 0,...,0,1,0,...\right)
\in G,$ the $n$th coordinate of which is 1 and the rest are zeros $\left(
n\in \mathbb{N}\right) .$

For $k\in \mathbb{N}$ and $x\in G$ denote 
\begin{equation*}
r_{k}\left( x\right) :=\left( -1\right) ^{x_{k}}
\end{equation*}%
the $k$th Rademacher function. If $n\in \mathbb{N}$, then $%
n=\sum\limits_{i=0}^{\infty }n_{i}2^{i}$ can be written, where $n_{i}\in
\{0,1\}$ $\left( i\in \mathbb{N}\right) $, i. e. $n$ is expressed in the
number system of base 2. Let us denote the order of $n$ by $\left\vert n\right\vert :=\max \{j\in 
\mathbb{N}\mathbf{:}n_{j}\neq 0\}$, that is $2^{\left\vert n\right\vert
}\leq n<2^{\left\vert n\right\vert +1}.$

The Walsh-Paley system is defined as the sequence of Walsh-Paley functions: 
\begin{equation*}
w_{n}\left( x\right) :=\prod\limits_{k=0}^{\infty }\left( r_{k}\left(
x\right) \right) ^{n_{k}}=r_{\left| n\right| }\left( x\right) \left(
-1\right) ^{\sum\limits_{k=0}^{\left| n\right| -1}n_{k}x_{k}}\quad\left(
x\in G,n\in \mathbf{P}\right) .
\end{equation*}

The Walsh-Kaczmarz functions are defined by $\kappa _{0}:=1$ and for $n\geq 1$ 
\begin{equation*}
\kappa _{n}(x):=r_{|n|}(x)\prod_{k=0}^{|n|-1}(r_{|n|-1-k}(x))^{n_{k}}.
\end{equation*}%
For $A\in \mathbf{\ }\mathbb{N}$ define the transformation 
$\tau_{A}\colon G\rightarrow G$ by 
\begin{equation*}
\tau _{A}(x):=(x_{A-1},x_{A-2},...,x_{0},x_{A},x_{A+1},...).
\end{equation*}%
By the definition of $\tau _{A}$ (see \cite{Skv}), we have 
\begin{equation*}
\kappa _{n}(x)=r_{|n|}(x)w_{n-2^{|n|}}(\tau _{|n|}(x))\quad (n\in \mathbf{\ }%
\mathbb{N},x\in G).
\end{equation*}

The Dirichlet kernels are defined by 
\begin{equation*}
D_{n}^{\alpha }(x):=\sum_{k=0}^{n-1}\alpha _{k}(x), \quad (n\in \mathbb N),
\end{equation*}%
where $\alpha _{k}=w_{k}$ (for all $k\in \mathbb{P}$) or $\kappa _{k}$ (for
all $k\in \mathbb{P}$). Recall that (see \cite{SWS}) 
\begin{equation}
D_{2^{n}}(x):=D_{2^{n}}^{w}(x)=D_{2^{n}}^{\kappa }(x)=%
\begin{cases}
2^{n}, & \text{ if }x\in I_{n}(0), \\ 
0, & \text{ if }x\notin I_{n}(0).
\end{cases}%
  \label{dir}
\end{equation}
\begin{equation}\label{dir2}
D_{n}^{w}\left( t\right) =w_{n}\left( t\right) \sum\limits_{j=0}^{\infty
}n_{j}w_{2^{j}}\left( t\right) D_{2^{j}}\left( t\right) ,  
\end{equation}%
where $n=\sum_{j=0}^{\infty }n_{j}2^{j}$. The $k$th partial sum of the Walsh(-Kaczmarz)-Fourier series of function $f$ at point $x$ is denoted by $S_k^\alpha(f;x)$. 

The Fej\'er kernels are defined as follows 
\begin{equation*}
K_{n}^{\alpha }(x):=\frac{1}{n}\sum_{k=0}^{n-1}D_{k}^{\alpha }(x).
\end{equation*}

The Kronecker product $\left( \alpha _{n,m}:n,m\in \mathbb{\ N}\right) $ of
two Walsh(-Kaczmarz) system is said to be the two-dimensional
Walsh(-Kaczmarz) system. Thus, 
\begin{equation*}
\alpha _{n,m}\left( x,y\right) =\alpha _{n}\left( x\right) \alpha _{m}\left(
y\right) .
\end{equation*}

If $f\in L_{1}(G^2),$ then the number $\hat{f}^{\alpha }\left( n,m\right)
:=\int\limits_{G^{2}}f\alpha _{n,m}$ $\left( n,m\in \mathbb{N}\right) $ is
said to be the $\left( n,m\right) $th Walsh-(Kaczmarz-)Fourier coefficient
of $f.$ Denote by $S_{n,m}^{\alpha }$ the $\left( n,m\right) $th partial sum
of the Walsh-(Kaczmarz-)Fourier series of a function $f$. Namely,

\begin{equation*}
S_{n,m}^{\alpha }(f;x,y):=\sum_{k=0}^{n-1}\sum\limits_{i=0}^{m-1}\hat{f}
^{\alpha }(k,i)\alpha _{k,i}(x,y).
\end{equation*}

Let us fix $d\geq 1,d\in \mathbb{P}$. For Walsh group $G$ let $G^{d}$ be its
Cartesian product $G\times \cdots \times G$ taken with itself $d$-times.

The norm (or quasinorm) of the space $L_{p}$ is defined by 
\begin{equation*}
\left\Vert f\right\Vert _{p}:=\left( \int\limits_{G^{2}}\left\vert f\left(
x,y\right) \right\vert ^{p}d\mu \left( x,y\right) \right) ^{1/p}\quad \left(
0<p<+\infty \right) .
\end{equation*}

\section{Best Approximation}

Denote by $E_{l,r}\left( f\right) $ the best approximation of a function $%
f\in C\left( G^{2}\right) $ by Walsh-Kaczmarz polynomials of degree $\leq l$
of a variable $x$ and of degree $\leq r$ of a variable $y$ and let $%
E_{l}^{\left( 1\right) }\left( f\right) $ be the partial best approximation
of  a function $f\in C\left( G^{2}\right) $ by Walsh-Kaczmarz polynomials
of degree $\leq l$ of a variable $x$, whose coefficients are continuous
functions of the remaining variable $y$, in
particular, best approximation with respect to polynomials $T_{l}^{\left(
1\right) }\left( x,y\right) :=\sum\limits_{j=0}^{l-1}\alpha _{j}\left(
y\right) \kappa_{j}\left( x\right)$. Analogously, we can define 
$E_{r}^{\left( 2\right) }\left( f\right) $.

Let $2^{L}\leq l<2^{L+1}$ and $E_{2^{L},2^{L}}\left( f\right) :=\left\Vert
f-T_{2^{L},2^{L}}\right\Vert _{C}$, 
where $E_{2^{L},2^{L}}\left( f\right)$ is the best approximation of 
$f\in C\left( G^{2}\right) $ by Walsh-Kaczmarz polynomials $T_{2^{L},2^{L}}$.

Since 
\begin{equation*}
\left\Vert S_{2^{L},2^{L}}\left( f\right) \right\Vert _{C}\leq \left\Vert
f\right\Vert _{C}
\end{equation*}%
we can write 
\begin{eqnarray}\label{B1}
\left\vert S_{l,l}^{\kappa}\left( f;x,y\right) -f\left( x,y\right) \right\vert
&\leq &\left\vert S_{l,l}^{\kappa}\left( f-S_{2^{L},2^{L}}\left( f\right)
;x,y\right) \right\vert +\left\Vert S_{2^{L},2^{L}}\left( f\right)
-f\right\Vert _{C}  \notag \\
&\leq &\left\vert S_{l,l}^{\kappa}\left( f-S_{2^{L},2^{L}}\left( f\right)
;x,y\right) \right\vert +\left\Vert S_{2^{L},2^{L}}\left(
f-T_{2^{L},2^{L}}\right) f\right\Vert _{C}  \notag \\
&&+\left\Vert f-T_{2^{L},2^{L}}f\right\Vert _{C}  \notag \\
&\leq &\left\vert S_{l,l}^{\kappa}\left( f-S_{2^{L},2^{L}}\left( f\right)
;x,y\right) \right\vert +2E_{2^{L},2^{L}}\left( f\right) .  \notag
\end{eqnarray}

It is well known that (see \cite{GGCA}) 
\begin{equation}
E_{2^{L},2^{L}}\left( f\right) \leq 2E_{2^{L}}^{\left( 1\right) }\left(
f\right) +2E_{2^{L}}^{\left( 2\right) }\left( f\right) .  \label{B2}
\end{equation}

It is easily seen that 
\begin{equation}
\left\Vert f-S_{2^{L},2^{L}}\left( f\right) \right\Vert _{C}\leq
2E_{2^{L},2^{L}}\left( f\right) .  \label{B4}
\end{equation}

\section{Main results}

\begin{theorem}
\label{T1}Let $f\in C\left( G^{2}\right) $. Then there exists a positive constant $c\left( f,A\right)$ depending only on $f$ and $A$ such that
the inequality 
\begin{eqnarray*}
&&\left\Vert \frac{1}{n}\sum\limits_{l=1}^{n}\left( e^{A\left\vert
S_{ll}^{k}\left( f\right) -f\right\vert ^{1/2}}-1\right) \right\Vert _{C} \\
&\leq &\frac{c\left( f,A\right) }{n}\sum\limits_{l=1}^{n}\left( \sqrt{%
E_{l}^{\left( 1\right) }\left( f\right) }+\sqrt{E_{l}^{\left( 2\right)
}\left( f\right) }\right) 
\end{eqnarray*}%
is satisfied for any $A>0.$
\end{theorem}

We say that the function $\psi $ belongs to the class $\Psi $ if it increases
on $[0,+\infty )$ and 
\begin{equation*}
\lim\limits_{u\rightarrow 0}\psi \left( u\right) =\psi \left( 0\right) =0.
\end{equation*}

\begin{theorem}\label{T2}
a)Let $\varphi \in \Psi $ and let the inequality 
\begin{equation}
\overline{\lim\limits_{u\rightarrow \infty }}\frac{\varphi \left( u\right) }{%
\sqrt{u}}<\infty  \label{T2-C}
\end{equation}%
hold. Then for any function $f\in C\left( G^{2}\right) $ the equality 
\begin{equation}
\lim\limits_{n\rightarrow \infty }\left\Vert \frac{1}{n}\sum%
\limits_{l=1}^{n}\left( e^{\varphi \left( \left\vert S_{ll}^{k}\left(
f\right) -f\right\vert \right) }-1\right) \right\Vert _{C}=0
\label{strongconv}
\end{equation}%
is satisfied.

b) For any function $\varphi \in \Psi $ satisfying the condition 
\begin{equation}
\overline{\lim\limits_{u\rightarrow \infty }}\frac{\varphi \left( u\right) }{%
\sqrt{u}}=\infty
\end{equation}%
there exists a function $F\in C\left( G^{2}\right) $ such that 
\begin{equation*}
\overline{\lim\limits_{m\rightarrow \infty }}\frac{1}{m}\sum%
\limits_{l=1}^{m}\left( e^{\varphi \left( \left\vert S_{ll}^{k}\left(
F;0,0\right) -f\left( 0,0\right) \right\vert \right) }-1\right) =+\infty .
\end{equation*}
\end{theorem}

\section{Auxiliary Results}
In this paper $c$ is a positive constant, which is not necessary the same at different occurrences.
\begin{lemma}
\label{gogoladze}(Gogoladze \cite{Go}) Let $\varphi ,\psi \in \Psi $ and the
equality 
\begin{equation*}
\lim\limits_{n\rightarrow \infty }\frac{1}{n}\sum\limits_{l=1}^{n}\psi
\left( \left\vert S_{l,l}^{\kappa}\left( f;x,y\right) -f\left( x,y\right)
\right\vert \right) =0
\end{equation*}%
be satisfied at the point $\left( x_{0},y_{0}\right) $ or uniformly on a set 
$E\subset G^{2}$. If 
\begin{equation*}
\overline{\lim\limits_{u\rightarrow \infty }}\frac{\varphi \left( u\right) }{%
\psi \left( u\right) }<\infty ,
\end{equation*}%
then the equality 
\begin{equation*}
\lim\limits_{n\rightarrow \infty }\frac{1}{n}\sum\limits_{l=1}^{n}\varphi
\left( \left\vert S_{l,l}^{\kappa}\left( f;x,y\right) -f\left( x,y\right)
\right\vert \right) =0
\end{equation*}%
is satisfied at the point $\left( x_{0},y_{0}\right) $ or uniformly on a set 
$E\subset G^{2}$.
\end{lemma}
Moreover, we will use the next Lemma of Glukhov \cite[p. 670]{Gl}.
\begin{lemma}[Glukhov \cite{Gl}]\label{lemma-GLU}
Let $\alpha_1,...,\alpha_n$ be real numbers. Let $p\in \mathbb{P}$ and $1<q\leq 2$. Then
$$
\frac{1}{n} \int\limits_{G^{p}} \left| \sum_{l=1}^n \alpha_l \prod\limits_{k=1}^{p}D_{l}^w\left(
x_{k}\right) \right| d\mathbf{\mu }\left( x_{1},...,x_{p}\right) 
\leq \frac{c}{{n}^{1/q}}\left( \sum_{k=1}^n |\alpha_k|^q
\right)^{1/q},
$$
where $c$ is depend only on $p$ and $q$.
\end{lemma}
In paper \cite[p. 672, l.
12-13]{Gl} it is stated that  constant $c$  depend on dimension and in dimension $p$ it will be $cp!$. Now, we choose $\alpha_k$ as special numbers in Lemma of Glukhov. Set   
\[
\alpha _{k}= 
\begin{cases}
1,& k=2^{n-1},...,2^{n}-1, \\ 
0,& \textrm{otherwise}.
\end{cases}
\]
We immediately have 

\begin{corollary} Let $p\in \mathbb{P}$. Then there exists an absolute constant $c$ such that
\begin{equation}\label{mainwalsh}
\sup\limits_{n}\int\limits_{G^{p}}\frac{1}{2^{n}}\left\vert
\sum\limits_{l=2^{n-1}}^{2^{n}-1}\prod\limits_{k=1}^{p}D_{l}^w\left(
x_{k}\right) \right\vert d\mathbf{\mu }\left( x_{1},...,x_{p}\right) 
\leq cp!.  
\end{equation}
\end{corollary}

\begin{lemma}\label{kaczmain}
There exists an absolute constant $c$ such that the
inequality 
\begin{equation}
\sup_{n}\int_{G^{d}}\frac{1}{2^{n}}\left\vert
\sum_{j=2^{n-1}}^{2^{n}-1}\prod_{k=1}^{d}D_{j}^{\kappa }(x_{k})\right\vert d%
\mathbf{\mu }\left( x_{1},...,x_{p}\right) \leq cd! 2^{d}  \label{main}
\end{equation}%
holds.
\end{lemma}

\begin{proof}
It is known (see Skvortsov \cite{Skv}) that 
\begin{equation*}
D_{2^{A}+j}^{\kappa }(x)=D_{2^{A}}(x)+r_{A}(x)D_{j}^{w}(\tau _{A}(x)),\quad
0\leq j<2^{A}.
\end{equation*}%
This implies
\begin{equation*}
\prod_{k=1}^{d}D_{2^{n-1}+j}^{\kappa }(x_{k})=\prod_{k=1}^{d}\left(
D_{2^{n-1}}(x_{k})+r_{n-1}(x_{k})D_{j}^{w}(\tau _{n-1}(x_{k}))\right)
\end{equation*}

\begin{equation*}
=\sum_{l=0}^{d}\sum_{\substack{ k_{1},\ldots ,k_{l}\in \{1,\ldots ,d\} \\ %
k_{r}\neq k_{s}\text{ if }r\neq s}}\prod_{m=1}^{l}D_{2^{n-1}}(x_{k_{m}})%
\prod_{k_{q}^{\prime }\in S_{d}^{l}}r_{n-1}(x_{k_{q}^{\prime
}})D_{j}^{w}(\tau _{n-1}(x_{k_{q}^{\prime }}))
\end{equation*}%
with the notation $S_{d}^{l}:=\{1,\ldots ,d\}\backslash \{k_{1},\ldots
,k_{l}\}$. That is, we have

\begin{equation*}
\left\vert \sum_{j=0}^{2^{n-1}-1}\prod_{k=1}^{d}D_{2^{n-1}+j}^{\kappa
}(x_{k})\right\vert \leq
\end{equation*}

\begin{eqnarray*}
&\leq &\sum_{l=0}^{d}\sum_{\substack{ k_{1},\ldots ,k_{l}\in \{1,\ldots ,d\} 
\\ k_{r}\neq k_{s}\text{ if }r\neq s}}\left\vert
\prod_{m=1}^{l}D_{2^{n-1}}(x_{k_{m}})\prod_{k_{q}^{\prime }\in
S_{d}^{l}}r_{n-1}(x_{k_{q}^{\prime }})\right. \\
&&\left. \times \sum_{j=0}^{2^{n-1}-1}\prod_{k_{q}^{\prime }\in
S_{d}^{l}}D_{j}^{w}(\tau _{n-1}(x_{k_{q}^{\prime }}))\right\vert
\end{eqnarray*}%
and
\begin{eqnarray*}
L_{n}&:=&\int_{G^{d}}\frac{1}{2^{n}}\left\vert
\sum_{j=2^{n-1}}^{2^{n}-1}\prod_{k=1}^{d}D_{j}^{\kappa }(x_{k})\right\vert
d\mu(x_{1})\ldots d\mu(x_{d})\\
&\leq &\sum_{l=0}^{d}\sum_{\substack{ k_{1},\ldots ,k_{l}\in \{1,\ldots ,d\}
\\ k_{r}\neq k_{s}\text{ if }r\neq s}} \\
&&\int_{G^{d-l}}\frac{1}{2^{n}}\left\vert
\sum_{j=0}^{2^{n-1}-1}\prod_{k_{q}^{\prime }\in S_{d}^{l}}D_{j}^{w}(\tau
_{n-1}(x_{k_{q}^{\prime }}))\right\vert d\mu(x_{k_{1}^{\prime }})\ldots
d\mu(x_{k_{d-l}^{\prime }}).
\end{eqnarray*}%
Since the transformation $\tau _{n-1}\colon G\rightarrow G$ is
measure-preserving \cite{Skv} and inequality \ref{mainwalsh} immediately
yields 
\begin{eqnarray*}
L_{n}&\leq&\sum_{l=0}^{d}\sum_{\substack{ k_{1},\ldots ,k_{l}\in \{1,\ldots ,d\}
\\ k_{r}\neq k_{s}\text{ if }r\neq s}} \\
&&\int_{G^{d-l}}\frac{1}{2^{n}}\left\vert
\sum_{j=0}^{2^{n-1}-1}\prod_{k_{q}^{\prime }\in S_{d}^{l}}D_{j}^{w}(x_{k_{q}^{\prime }})\right\vert d\mu(x_{k_{1}^{\prime }})\ldots
d\mu(x_{k_{d-l}^{\prime }})\\
 &\leq & c
\sum_{l=0}^{d}\sum_{\substack{ k_{1},\ldots ,k_{l}\in \{1,\ldots
,d\} \\ k_{r}\neq k_{s}\text{ if }r\neq s}}(d-l)!
\leq d! \sum_{l=0}^{d}\sum_{\substack{ k_{1},\ldots ,k_{l}\in \{1,\ldots
,d\} \\ k_{r}\neq k_{s}\text{ if }r\neq s}} 1.
\end{eqnarray*}
Since, the number of all subsets of the set $\{ 1,\ldots ,d\}$ is $2^d$, we immediately have 
$$
L_n\leq cd!2^d.
$$
Taking the supremum for all $n\in {\mathbb{N}}$ completes the proof of 
Lemma \ref{kaczmain}.
\end{proof}

\begin{lemma}
\label{mainlemma}Let $p>0$. Then 
\begin{equation}
\left\{ \frac{1}{2^{A}}\sum\limits_{l=2^{A}}^{2^{A+1}-1}\left\vert
S_{l,l}^{\kappa}\left( f;x,y\right) \right\vert ^{p}\right\} ^{1/p}\leq
c\left\Vert f\right\Vert _{C}\left( p+1\right) ^{2}.  \label{mainineq}
\end{equation}
\end{lemma}

\begin{proof}
Since
\begin{eqnarray*}
\left\{ \frac{1}{2^{A}}\sum\limits_{l=2^{A}}^{2^{A+1}-1}\left\vert
S_{l,l}^{\kappa}\left( f;x,y\right) \right\vert ^{p}\right\} ^{1/p} 
&\leq &\left\{ \frac{1}{2^{A}}\sum\limits_{l=2^{A}}^{2^{A+1}-1}\left\vert
S_{l,l}^{\kappa}\left( f;x,y\right) \right\vert ^{p+1}\right\} ^{1/\left(
p+1\right) }
\end{eqnarray*}%
without lost of generality we can suppose that $p=2^{m},m\in \mathbb{P}$. We
can write%
\begin{equation*}
\left\vert S_{l,l}^{\kappa}\left( f;x,y\right) \right\vert ^{2}=S_{l,l}^{\kappa}\left(
f;x,y\right) S_{l,l}^{\kappa}\left( f;x,y\right) 
\end{equation*}%
\begin{equation*}
=\int\limits_{G^{2}}f\left( x+s_{1},y+t_{1}\right) D_{l}^{\kappa}\left(
s_{1}\right) D_{l}^{\kappa}\left( t_{1}\right) d\mathbf{\mu }\left(
s_{1},t_{1}\right) 
\end{equation*}%
\begin{equation*}
\times \int\limits_{G^{2}}f\left( x+s_{2},y+t_{2}\right) D_{l}^{\kappa}\left(
s_{2}\right) D_{l}^{\kappa}\left( t_{2}\right) d\mathbf{\mu }\left(
s_{2},t_{2}\right) 
\end{equation*}%
\begin{equation*}
=\int\limits_{G^{4}}f\left( x+s_{1},y+t_{1}\right) f\left(
x+s_{2},y+t_{2}\right) 
\end{equation*}%
\begin{equation*}
\times D_{l}^{\kappa}\left( s_{1}\right) D_{l}^{\kappa}\left( s_{2}\right)
D_{l}^{\kappa}\left( t_{1}\right) D_{l}^{\kappa}\left( t_{2}\right) d\mathbf{\mu }%
\left( s_{1},t_{1},s_{2},t_{2}\right) .
\end{equation*}%
Hence from Lemma \ref{kaczmain}, we get%
\begin{equation*}
\left\vert S_{l,l}^{\kappa}\left( f;x,y\right) \right\vert ^{p}=\left( \left\vert
S_{l,l}^{\kappa}\left( f;x,y\right) \right\vert ^{2}\right) ^{p/2}
\end{equation*}%
\begin{equation*}
=\int\limits_{G^{2p}}\prod\limits_{k=1}^{p}f\left( x+s_{k},y+t_{k}\right) 
 \prod\limits_{i=1}^{p}D_{l}^{\kappa}\left( s_{i}\right)
\prod\limits_{j=1}^{p}D_{l}^{\kappa}\left( t_{j}\right) d\mathbf{\mu }\left(
s_{1},t_{1},...,s_{p},t_{p}\right) ,
\end{equation*}
and
\begin{equation*}
\left\{ \frac{1}{2^{A}}\sum\limits_{l=2^{A}}^{2^{A+1}-1}\left\vert
S_{l,l}^{\kappa}\left( f;x,y\right) \right\vert ^{p}\right\} ^{1/p}
\end{equation*}%
\begin{equation*}
\leq \left( \int\limits_{G^{2p}}\prod\limits_{k=1}^{p}f\left(
x+s_{k},y+t_{k}\right) \right. 
\end{equation*}%
\begin{equation*}
\left. \times \frac{1}{2^{A}}\left\vert
\sum\limits_{l=2^{A}}^{2^{A+1}-1}\prod\limits_{i=1}^{p}D_{l}^{\kappa}\left(
s_{i}\right) \prod\limits_{j=1}^{p}D_{l}^{\kappa}\left( t_{j}\right) \right\vert d%
\mathbf{\mu }\left( s_{1},t_{1},...,s_{p},t_{p}\right) \right) ^{1/p}
\end{equation*}%
\begin{equation*}
\leq \left\Vert f\right\Vert _{C}\left( \int\limits_{G^{2p}}\frac{1}{2^{A}}%
\left\vert
\sum\limits_{l=2^{A}}^{2^{A+1}-1}\prod\limits_{i=1}^{p}D_{l}^{\kappa}\left(
s_{i}\right) \prod\limits_{j=1}^{p}D_{l}^{\kappa}\left( t_{j}\right) \right\vert d%
\mathbf{\mu }\left( s_{1},t_{1},...,s_{p},t_{p}\right) \right) ^{1/p}
\end{equation*}%
\begin{equation*}
\leq cp^{2}\left\Vert f\right\Vert _{C}.
\end{equation*}

Lemma \ref{mainlemma} is proved.
\end{proof}

\begin{lemma}\label{BAest}
Let $f\in C\left( G^{2}\right) $ and $p>0.$ Then 
\begin{eqnarray} \label{bestestimation} 
&&\frac{1}{n}\sum\limits_{l=1}^{n}\left\vert S_{l,l}^{\kappa}\left( f;x,y\right)
-f\left( x,y\right) \right\vert ^{p} \\
&\leq &c^{p}\cdot \left( p+1\right) ^{2p}\left\{ \frac{1}{n}%
\sum\limits_{l=1}^{n}\left( E_{l}^{\left( 1\right) }\left( f\right) \right)
^{p}+\frac{1}{n}\sum\limits_{r=1}^{n}\left( E_{r}^{\left( 2\right) }\left(
f\right) \right) ^{p}\right\} .  \notag
\end{eqnarray}
\end{lemma}
\begin{proof}Since
\begin{equation*}
\left( a+b\right) ^{\beta }\leq 2^{\beta }\left( a^{\beta }+b^{\beta
}\right) ,\beta >0
\end{equation*}%
using (\ref{B1})-(\ref{B4}) and  Lemma \ref{mainlemma} we get 
\begin{equation}
\frac{1}{2^{A}}\sum\limits_{l=2^{A}}^{2^{A+1}-1}\left\vert S_{l,l}^{\kappa}\left(
f;x,y\right) -f\left( x,y\right) \right\vert ^{p}  \label{above}
\end{equation}%
\begin{equation*}
\leq \frac{2^{p}}{2^{A}}\sum\limits_{l=2^{A}}^{2^{A+1}-1}\left\vert
S_{l,l}^{\kappa}\left( f-S_{2^{A},2^{A}}\left( f\right) ;x,y\right) \right\vert
^{p}+2^{2p}E_{2^{A},2^{A}}^{p}\left( f\right) 
\end{equation*}
\begin{equation*}
\leq c^{p}\left( p+1\right) ^{2p}\left\Vert f-S_{2^{A},2^{A}}\left( f\right)
\right\Vert _{C}^{p}+c^{p}E_{2^{A},2^{A}}^{p}\left( f\right)
\end{equation*}%
\begin{equation*}
\leq c^{p}\left( p+1\right) ^{2p}\left( \left( E_{2^{A}}^{\left( 1\right)
}\left( f\right) \right) ^{p}+\left( E_{2^{A}}^{\left( 2\right) }\left(
f\right) \right) ^{p}\right) .
\end{equation*}

Let $2^{N}\leq n<2^{N+1}$ . Then from (\ref{above}) we have%
\begin{equation*}
\frac{1}{n}\sum\limits_{l=1}^{n}\left\vert S_{l,l}^{\kappa}\left( f;x,y\right)
-f\left( x,y\right) \right\vert ^{p}
\end{equation*}%
\begin{equation*}
\leq \frac{1}{n}\sum\limits_{l=1}^{2^{N+1}-1}\left\vert S_{l,l}^{\kappa}\left(
f;x,y\right) -f\left( x,y\right) \right\vert ^{p}
\end{equation*}%
\begin{equation*}
=\frac{1}{n}\sum\limits_{A=0}^{N}\sum\limits_{l=2^{A}}^{2^{A+1}-1}\left\vert
S_{l,l}^{\kappa}\left( f;x,y\right) -f\left( x,y\right) \right\vert ^{p}
\end{equation*}%
\begin{equation*}
\leq \frac{c^{p}\left( p+1\right) ^{2p}}{n}\sum\limits_{A=0}^{N}2^{A}\left(
\left( E_{2^A}^{\left( 1\right) }\left( f\right) \right) ^{p}+\left(
E_{2^A}^{\left( 2\right) }\left( f\right) \right) ^{p}\right) 
\end{equation*}%
\begin{equation*}
\leq \frac{c^{p}\left( p+1\right) ^{2p}}{n}\sum\limits_{A=1}^{N}\sum%
\limits_{l=2^{A-1}}^{2^{A}-1}\left( \left( E_{l}^{\left( 1\right) }\left(
f\right) \right) ^{p}+\left( E_{l}^{\left( 2\right) }\left( f\right) \right)
^{p}\right) 
\end{equation*}%
\begin{equation*}
\leq c^{p}\cdot \left( p+1\right) ^{2p}\left\{ \frac{1}{n}%
\sum\limits_{l=1}^{n}\left( E_{l}^{\left( 1\right) }\left( f\right) \right)
^{p}+\frac{1}{n}\sum\limits_{r=1}^{n}\left( E_{r}^{\left( 2\right) }\left(
f\right) \right) ^{p}\right\} .
\end{equation*}

Lemma \ref{BAest} is proved.
\end{proof}

\section{Proofs of Main Results}

The Walsh-Paley version of Theorem \ref{T1} were proved in \cite{GGCA}.
Based on inequality (\ref{bestestimation}) the same construction works for
the Walsh-Kaczmarz case. Therefore the proof of Theorem \ref{T1} will be
omitted.

\begin{proof}[Proof of Theorem 2]
a) It is easily seen  that if $\varphi \in \Psi $, then $e^{\varphi }-1\in
\Psi .$ Besides, (\ref{T2-C}) implies the existence of a number $A$ such
that 
\begin{equation*}
\overline{\lim\limits_{u\rightarrow \infty }}\frac{e^{\varphi \left(
u\right) }-1}{e^{Au^{1/2}}-1}<\infty .
\end{equation*}%
Therefore, in view of Lemma \ref{gogoladze}, to prove Theorem \ref{T2} it is sufficient to show that 
\begin{equation}
\lim\limits_{n\rightarrow \infty }\left\Vert \frac{1}{n}\sum%
\limits_{l=1}^{n}\left( e^{A\left\vert S_{l,l}^{\kappa}\left( f\right)
-f\right\vert ^{1/2}}-1\right) \right\Vert _{C}=0.  \label{=0}
\end{equation}%
The validity of equality (\ref{=0}) immediately follows from Theorem \ref{T1}%
.

b) Such a construction for the analogical problem has already been made for the Walsh-Paley case \cite{GGCA}, where Walsh-Paley function defined on [0,1]. We will use idea from above mentioned paper and we construct similar, but not the same function for Walsh-Kaczmarz case, where Walsh-Kaczmarz system is defined on Walsh group. The common aspect of two 
construction is stated in the inequality \eqref{5Ak}, later.

First of all, let us prove the validity of point b) in the one-dimensional
case. In particular, we prove that if $\psi \in \Psi $ satisfying the condition 
\begin{equation*}
\overline{\lim\limits_{u\rightarrow \infty }}\frac{\psi \left( u\right) }{{u}%
}=\infty,
\end{equation*}%
then there exists a function $f\in C\left( G\right) $ such that 
\begin{equation}
\overline{\lim\limits_{m\rightarrow \infty }}\frac{1}{m}\sum%
\limits_{l=1}^{m}\left( e^{\psi \left( \left\vert S_{l}^{\kappa}\left( f;0\right)
-f\left( 0\right) \right\vert \right) }\right) =+\infty .  \label{1-dim}
\end{equation}

Let $\left\{ B_{k}:k\geq 1\right\} $ be an increasing sequence of positive
integers such that%
\begin{equation}\label{Bk>}
B_{k}>2B_{k-1},
\end{equation}%
\begin{equation}
\frac{\psi \left( B_{k}\right) }{B_{k}}>\frac{5k}{c^{\prime }}\text{%
\thinspace \thinspace ,}  \label{>160}
\end{equation}%
where $c^{\prime }$ will be defined later.

Set 
\begin{equation*}
A_{k}:=\left[ \frac{k}{c^{\prime }}B_{k}\right]
\end{equation*}%
and 
\begin{equation*}
N_{A_{k}}:=2^{2A_{k}}+2^{2A_{k}-2}+\cdots +2^{2}+2^{0},
\end{equation*}

Set%
\begin{equation*}
f_{j}\left( x\right) :=\frac{1}{j+1}\sum\limits_{l=A_{j-1}}^{A_{j}-1}\sum%
\limits_{x_{0}=0}^{1}\cdots \sum\limits_{x_{2A_{j}-2l-1}=0}^{1}\text{sgn}%
\left( D_{N_{A_{j}}}^{\kappa}\left( x\right) \right)
\end{equation*}%
\begin{equation*}
\times \mathbb{I}_{I_{2A_{j}+2}\left(
x_{0},...,x_{2A_{j}-2l-1},x_{2A_{j}-2l}=1,0,...,0\right) }\left( x\right) ,
\end{equation*}%
\begin{equation*}
f\left( x\right) :=\sum\limits_{j=0}^{\infty }f_{j}\left( x\right) ,\quad 
f\left(
0\right) =0,
\end{equation*}%
where $\mathbb{I}_{E}$ is characteristic function of the set $E\subset G_{m}$%
.

It is easily seen that $f\in C\left( G\right) $.

We can write%
\begin{equation}\label{J1-J3}
\left\vert S_{N_{A_{k}}}^{\kappa}\left( f;0\right) -f\left( 0\right) \right\vert
=\left\vert S_{N_{A_{k}}}^{\kappa}\left( f;0\right) \right\vert   
\end{equation}%
\begin{equation*}
=\left\vert \int\limits_{G}f\left( t\right) D_{N_{A_{k}}}^{\kappa }\left(
t\right) d\mu \left( t\right) \right\vert 
\end{equation*}%
\begin{equation*}
\geq \left\vert \int\limits_{G}f_{k}\left( t\right) D_{N_{A_{k}}}^{\kappa
}\left( t\right) d\mu \left( t\right) \right\vert 
\end{equation*}%
\begin{equation*}
-\sum\limits_{j=k+1}^{\infty }\left\vert \int\limits_{G}f_{j}\left( t\right)
D_{N_{A_{k}}}^{\kappa }\left( t\right) d\mu \left( t\right) \right\vert 
\end{equation*}%
\begin{equation*}
-\sum\limits_{j=0}^{k-1}\left\vert \int\limits_{G}f_{j}\left( t\right)
D_{N_{A_{k}}}^{\kappa }\left( t\right) d\mu \left( t\right) \right\vert 
\end{equation*}%
\begin{equation*}
=J_{1}-J_{2}-J_{3}.
\end{equation*}

From the definition of the function $f$ we have%
\begin{equation}
J_{1}=\frac{1}{k+1}\sum\limits_{l=A_{k-1}}^{A_{k}-1}\sum%
\limits_{t_{0}=0}^{1}\cdots \sum\limits_{t_{2A_{k}-2l-1}=0}^{1}  \label{J1-1}
\end{equation}%
\begin{equation*}
\int\limits_{_{I_{2A_{k}+2}\left(
t_{0},...,t_{2A_{k}-2l-1},t_{2A_{k}-2l}=1,0,...,0\right) }}\left\vert
D_{N_{A_{k}}}^{\kappa }\left( t\right) \right\vert d\mu \left( t\right) 
\end{equation*}%
Since (see Skvortsov \cite{Skv})%
\begin{equation}
D_{N_{A_{k}}}^{\kappa }\left( t\right) =D_{2^{2A_{k}}}^{w}\left( t\right)
+r_{2A_{k}}\left( t\right) D_{N_{A_{k}-1}}^{w}\left( \tau _{2A_{k}}\left(
t\right) \right)   \label{D-K}
\end{equation}%
\begin{equation*}
=r_{2A_{k}}\left( t\right) D_{N_{A_{k}-1}}^{w}\left( \tau _{2A_{k}}\left(
t\right) \right) ,
\end{equation*}%
we can write%
\begin{equation*}
\left\vert D_{N_{A_{k}}}^{\kappa }\left( t\right) \right\vert =\left\vert
D_{N_{A_{k}-1}}^{w}\left( \tau _{2A_{k}}\left( t\right) \right) \right\vert 
\end{equation*}%
\begin{equation*}
=\left\vert \sum\limits_{j=0}^{l}r_{2j}\left( \tau _{2A_{k}}\left( t\right)
\right) D_{2j}\left( \tau _{2A_{k}}\left( t\right) \right) \right\vert 
\end{equation*}%
\begin{equation*}
\geq 2^{2l}-\sum\limits_{j=0}^{l-1}2^{2j}\geq c2^{2l},
\end{equation*}%
\begin{equation*}
t\in I_{2A_{k}+2}\left(
t_{0},...,t_{2A_{k}-2l-1},t_{2A_{k}-2l}=1,0,...,0\right) .
\end{equation*}%
Hence from (\ref{Bk>}) and (\ref{J1-1}) we have%
\begin{equation}
J_{1}\geq \frac{c}{k}\sum\limits_{l=A_{k-1}}^{A_{k}-1}\sum%
\limits_{t_{0}=0}^{1}\cdots \sum\limits_{t_{2A_{k}-2l-1}=0}^{1}\frac{2^{2l}}{%
2^{2A_{k}}}  \label{J1}
\end{equation}%
\begin{equation*}
\geq \frac{c}{k}\sum\limits_{l=A_{k-1}}^{A_{k}-1}\frac{2^{2l}2^{2A_{k}-2l}}{%
2^{2A_{k}}}\geq \frac{c\left( A_{k}-A_{k-1}\right) }{k}\geq \frac{c_{0}A_{k}%
}{k}.
\end{equation*}

For $J_{2}$ we have%
\begin{equation}
J_{2}\leq c\sum\limits_{j=k+1}^{\infty }\frac{1}{j+1}\sum%
\limits_{l=A_{j-1}}^{A_{j}-1}\frac{2^{2A_{j}-2l}}{2^{2A_{j}}}N_{A_{k}}
\label{J2}
\end{equation}%
\begin{equation*}
\leq \frac{cN_{A_{k}}}{k}\sum\limits_{l=A_{k}}^{\infty }\frac{1}{2^{l}}\leq 
\frac{c}{k}.
\end{equation*}

By (\ref{dir}) and from the construction of the function $f_{j}$ we can
write 
\begin{equation*}
\int\limits_{G}f_{j}\left( t\right) D_{N_{A_{k}}}^{\kappa }\left( t\right)
d\mu \left( t\right) 
\end{equation*}%
\begin{equation*}
=\int\limits_{G}f_{j}\left( t\right) r_{2A_{k}}\left( t\right)
D_{N_{A_{k}-1}}^{w}\left( \tau _{2A_{k}}\left( t\right) \right) d\mu \left(
t\right) =0,
\quad j=1,2,...,k-1
\end{equation*}%
consequently%
\begin{equation}
J_{3}=0.  \label{J3}
\end{equation}

Combining (\ref{>160})-(\ref{J3}) we conclude that%
\begin{equation}\label{5Ak}
\left\vert S_{N_{A_{k}}}^{\kappa}\left( f;0\right) \right\vert =\left\vert
S_{N_{A_{k}}}^{\kappa}\left( f;0\right) -f\left( 0\right) \right\vert \geq \frac{%
c^{\prime }A_{k}}{k}=B_{k},  
\end{equation}%
\begin{equation*}
\psi \left( \left\vert S_{N_{A_{k}}}^{\kappa}\left( f;0\right) \right\vert
\right) \geq \psi \left( B_{k}\right) \geq \frac{5k}{c^{\prime }}B_{k}\geq
5A_{k}.
\end{equation*}
We note that for
Walsh-Fourier series function with properties \eqref{5Ak} was constructed in \cite{GGCA}.
The construction in \cite{GGCA} is given for  $[0,1)$ interval. 

Let us write $\varphi \left( u\right) =\lambda \left( u\right) \sqrt{u}$ and define 
$\psi \left( u\right) :=\lambda \left( u^{2}\right) u.$ Then%
\begin{equation*}
\lim\limits_{u\rightarrow \infty }\frac{\psi \left( u\right) }{u}=+\infty .
\end{equation*}%
Therefore (see \eqref{5Ak}) there exists a function $f\in C\left( G\right) $
for which%
\begin{equation}
\psi \left( \left\vert S_{N_{A_{k}}}^{\kappa}\left( f,0\right) \right\vert
\right) \geq 5A_{k}.  \label{ineq2}
\end{equation}%
Set%
\begin{equation*}
F\left( x,y\right) :=f\left( x\right) f\left( y\right) .
\end{equation*}%
It is easy to show that%
\begin{eqnarray*}
\varphi \left( \left\vert S_{N_{A_{k}},N_{A_{k}}}^{\kappa}\left( F;0,0\right)
\right\vert \right) 
&=&\varphi \left( \left\vert S_{N_{A_{k}}}^{\kappa}\left( f;0\right) \right\vert
^{2}\right)  \\
&=&\lambda \left( \left\vert S_{N_{A_{k}}}^{\kappa}\left( f;0\right) \right\vert
^{2}\right) \left\vert S_{N_{A_{k}}}^\kappa\left( f;0\right) \right\vert  \\
&=&\psi \left( \left\vert S_{N_{A_{k}}}^{\kappa}\left( f;0\right) \right\vert
\right) .
\end{eqnarray*}%
Consequently, from (\ref{ineq2})  we have%
\begin{eqnarray*}
\frac{1}{N_{A_{k}}}\sum\limits_{i=1}^{N_{A_{k}}}e^{\varphi \left(
\left\vert S_{i,i}^\kappa\left( F;0,0\right) \right\vert \right) } 
&\geq &\frac{1}{N_{A_{k}}}e^{\varphi \left( \left\vert
S_{N_{A_{k}},N_{A_{k}}}^\kappa\left( F;0,0\right) \right\vert \right) } \\
&=&\frac{1}{N_{A_{k}}}e^{\psi \left( \left\vert S_{N_{A_{k}}}^\kappa\left(
f;0\right) \right\vert \right) } \\
&\geq &\frac{e^{5A_{k}}}{2^{2A_{k}}}\rightarrow \infty \text{ as~\ }%
k\rightarrow \infty .
\end{eqnarray*}%
Theorem \ref{T2} is proved.
\end{proof}

\end{document}